\documentclass[11pt]{amsart}

\usepackage{amssymb,amscd,amsmath,hyperref,color,enumerate,tikz-cd}
\usepackage[all,cmtip]{xy}
\usepackage{mathrsfs}

\usepackage[margin=3cm]{geometry}

\newcommand{\R}{{\ensuremath{\mathbb{R}}}}
\newcommand{\N}{{\ensuremath{\mathbb{N}}}}

\newcommand{\C}{{\ensuremath{\mathbb{C}}}}
\newcommand{\B}{{\ensuremath{\mathbb{B}}}}

\def\int{\mathop{\mathrm{Int}}}

\DeclareMathOperator{\diag}{diag}
\DeclareMathOperator{\Tr}{Tr}

\newcommand{\Sph}{\ensuremath{\mathbb{S}}}

\newcommand{\abs}[1]{\ensuremath{ {\left| #1 \right|} }}
\newcommand{\inner}[2]{\ensuremath{ {\left< #1 , #2 \right>} }}
\newcommand{\norm}[1]{\ensuremath{ {\left\| #1 \right\|} }}

\newtheorem{proposition}{Proposition}[section]
\newtheorem{lemma}[proposition]{Lemma}

\newtheorem{theorem}[proposition]{Theorem}
\newtheorem{corollary}[proposition]{Corollary}
\theoremstyle{definition}

\newtheorem{problem}[proposition]{Problem}

\newtheorem{example}[proposition]{Example}

\newtheorem{claim}{Claim}[proposition]

\numberwithin{equation}{section}

\begin{document}


\title{A variant of the Kaplansky problem for maps on positive matrices}

 \author{Mateo Toma\v{s}evi\'{c}}


\address{M.~Toma\v{s}evi\'c, Department of Mathematics, Faculty of Science, University of Zagreb, Bijeni\v{c}ka 30, 10000 Zagreb, Croatia}
\email{mateo.tomasevic@math.hr}

\thanks{The author thanks Peter \v{S}emrl for introducing the problem and for providing an outline of a possible approach. The author is also very grateful to Ilja Gogi\'{c} for many helpful suggestions and ideas.}



\date{\today}

\begin{abstract}
We prove that all injective maps on positive complex matrices which preserve order and shrink spectrum are implemented by unitary or antiunitary conjugations. We show by counterexamples that all assumptions are indispensable. The result easily generalizes to maps on hermitian matrices.
\end{abstract}

\maketitle

\section{Introduction}

By the well-known Gleason-Kahane-\.{Z}elazko theorem (\cite{Gleason, KahaneZelazko}), a linear functional $\varphi : A \to \C$ on a unital complex Banach algebra $A$ is a character (i.e. a nonzero algebra homomorphism) if and only if $\varphi$ maps every element $a \in A$ inside its spectrum $\sigma(a)$. One can also easily show (\cite[Lemma 2.1.1]{Kanuith})  that a linear functional $\varphi : A \to \C$ is a character if and only if $\varphi$ is a unital Jordan homomorphism, that is $\varphi(1) = 1$ and $\varphi$ preserves squares (i.e. $\varphi(a^2) = \varphi(a)^2$ for all $a \in A$).





Inspired by the Gleason-Kahane-\.{Z}elazko theorem, Kaplansky in 1970 formulated the following problem within his famous lecture notes \cite{Kaplansky}:

\begin{problem}[The Kaplansky problem, the general variant]
Let $\phi : A \to B$ be a linear unital map between complex unital Banach algebras which shrinks spectrum, i.e. it satisfies $\sigma(\phi(a)) \subseteq \sigma(a)$ for all $a \in A.$ Is $\phi$ necessarily a Jordan homomorphism?
\end{problem}

Note that since $\phi$ is linear and unital, the spectrum shrinking condition is equivalent to the simpler requirement that $\phi$ preserves invertibility, i.e. for every invertible element $a \in A$, $\phi(a)$ is invertible in $B$. It is well-known that the general variant of the Kaplansky problem has a negative answer. Namely, consider the unital subalgebra of block upper triangular matrices
$$A = \left\{\begin{bmatrix} A & B \\ 0 & C\end{bmatrix} \in M_4(\C) : A,B,C \in M_2(\C)\right\}$$
of the unital Banach algebra $M_4(\C)$, and the linear map
$$\phi : A \to A, \qquad \phi\left(\begin{bmatrix} A & B \\ 0 & C\end{bmatrix}\right) := \begin{bmatrix} A & B \\ 0 & C^t\end{bmatrix}$$
which is unital and spectrum preserving, but not a Jordan homomorphism (see \cite[page 28]{Aupetit1}). Note that the above Banach algebra $A$ is not semisimple.

A famous conjecture due to Aupetit states that the Kaplansky problem has a positive answer under additional assumptions that the Banach algebras $A$ and $B$ are semisimple and that $\phi$ is surjective. As far as the author knows, this problem is still widely open, even for $C^*$-algebras (see \cite{BresarSemrl}, page 270).

In this paper, we consider a variant of the Kaplansky problem which removes the linearity assumption and restricts the problem to the cone of positive elements $\B(H)_+$ of a particular $C^*$-algebra $\B(H)$ of bounded linear operators on a Hilbert space $H$. Direct inspiration for this paper is the main result from \cite{semrl}, which also easily generalizes to maps on self-adjoint operators:

\begin{theorem}[\v{S}emrl]\label{semrl}
Let $H$ be a Hilbert space and $\phi : \B(H)_+ \to \B(H)_+$ a surjective map which preserves order and spectrum. Then there exists a unitary or antiunitary map $U : H \to H$ such that
$$\phi(A) = UAU^*, \quad \text{ for all }A \in \B(H)_+.$$
\end{theorem}

Recall that an antiunitary map $U : H \to H$ is defined as an antilinear map such that for all $x,y \in H$ we have $\inner{Ux}{Uy} = \overline{\inner{x}{y}}$. Every such map has an autiunitary adjoint $U^* : H \to H$ uniquely characterized by the property $\inner{Ux}{y} = \overline{\inner{x}{U^*y}}$ for all $x,y \in H$.

We originally considered an analogous problem to Theorem \ref{semrl} with the condition of surjectivity reduced to some sort of local surjectivity. This can be achieved by restricting the problem to finite dimensional spaces and requiring injectivity and continuity of the map $\phi$, whence the invariance of domain theorem can be applied providing that $\phi$ is a homeomorphism on every open subset. We then noticed that it suffices to require spectrum shrinking instead of spectrum preserving; the same result follows since $\phi^{-1}$ has to preserve spectrum on the dense set of matrices with $n$ distinct eigenvalues so it was easy to show that $\phi$ in fact preserves spectrum everywhere.

However, when trying to show optimality of such a result via counterexamples, we noticed that continuity can in fact be shown from other assumptions, which yielded the final version of our main result:

\begin{theorem}\label{main result}
Let $H_n^{\ge 0}$ be the set of $n\times n$ positive complex matrices and let $\phi : H_n^{\ge 0} \to H_n^{\ge 0}$ be an injective map which preserves order and shrinks spectrum. Then there exists a unitary or antiunitary map $U : \C^n \to \C^n$ such that
$$\phi(A) = UAU^*, \quad \text{ for all }A \in H_n^{\ge 0}.$$
\end{theorem}

When dealing with maps on hermitian matrices, antiunitary maps can be altogether avoided. Indeed, we can reformulate the conclusion of Theorem \ref{main result} in the spirit of matrices only: there exists a unitary matrix $U \in U_n$ such that $\phi(A) = UAU^*, \forall A \ge 0$ or $\phi(A) = UA^tU^*, \forall A \ge 0$.


We also note that after showing continuity of the map $\phi$, the use of the invariance of domain theorem remains a crucial point of the proof. Indeed, one can wonder whether an analogous result can be formulated on infinite dimensional spaces, akin to Theorem \ref{semrl}. This turns out to be false even if we assume that $\phi$ is spectrum preserving. Indeed, let $H$ be an infinite-dimensional Hilbert space. Then $H \cong H \oplus H$, so let $U : H \to H \oplus H$ be a unitary map. Let $\varphi : \B(H)_+ \to \B(H)_+$ be an arbitrary map which preserves order and shrinks spectrum (a nice example is $\varphi(X) := \norm{X}I$). Define
$$\phi : \B(H)_+ \to \B(H)_+, \quad \phi(A) := U^*(A \oplus \varphi(A))U, \quad \text{ for all }A \ge 0.$$
Then $\phi$ is an injective map which preserves order and spectrum, but is not of the form as stated in Theorem \ref{main result}. Indeed, $\phi$ is clearly not surjective since the image of $U\phi(\cdot)U^*$ is contained in the set of diagonal operators on $H \oplus H$.

Moreover, the above map $\phi$ preserves order in both directions, i.e. satisfies $A \le B \iff \phi(A) \le \phi(B)$. In this context, it is interesting  that a surjective map $\phi : \B(H)_+ \to \B(H)_+$ where $\dim H > 1$ which preserves order in both directions is automatically of the form given in Theorem \ref{semrl} (\cite[Theorem 2.5.1]{Molnar}), even without assuming that $\phi$ preserves spectrum. Note that injectivity is immediate from preserving order in both directions, but as the above example shows, injectivity alone is not sufficient, even if we assume that $\phi$ preserves spectrum.

\section{Preliminaries}\label{sec:prel}

 We begin this section by introducing some notation and terminology. Let $n \in\N$.
\begin{itemize}
    \item $M_n := M_n(\C)$ denotes the set of all $n\times n$ complex matrices,
    \item $H_n$ denotes the set of hermitian matrices in $M_n$,
    \item $H_n^{\ge 0}$ denotes the set of positive matrices in $M_n$,
    \item $H_n^{> 0}$ denotes the set of strictly positive matrices in $M_n$,
    \item $U_n$ denotes the set of unitary matrices in $M_n$,
    \item $\mathcal{P}$ denotes the set of orthogonal projections in $M_n$,
    \item $\mathcal{P}_k$ denotes the set of orthogonal projections in $M_n$ of rank $k$,
    \item For $A,B \in M_n$ we say that $A$ is perpendicular to $B$ (and write $A \perp B$) if $AB = BA = 0$,
    \item For $A \in M_n$ by $r(A)$ we denote the rank of $A$,
    \item For $A \in M_n$ by $\nu(A)$ we denote the spectral radius of $A$,
    \item For $A \in H_n$ by $\lambda_{\mathrm{min}}(A)$ and $\lambda_{\mathrm{max}}(A)$ we denote the smallest and largest eigenvalue of $A$, respectively,
    \item For $A \in M_n$ we denote by $R(A)$ the image of $A$ and by $N(A)$ the nullspace of $A$.
\end{itemize}

Of great relevance for our proof is the Monotonicity theorem for hermitian matrices, proof of which can be found in \cite[Corollary 4.3.12]{HornJohnson}:

\begin{theorem}[Monotonicity theorem]\label{monotonicity theorem}
Let $A, B \in H_n$ and denote their eigenvalues in increasing order by $\lambda_1(A), \ldots, \lambda_n(A)$ and $\lambda_1(B), \ldots, \lambda_n(B)$ respectively. Then if $A \le B$ we have
$$\lambda_j(A) \le \lambda_j(B), \quad \text{ for all $1 \le j \le n$}.$$
If $A < B$, all the inequalities are strict.
\end{theorem}

We also need the next simple lemma:

\begin{lemma}\label{projection equality}
Let $P,Q \in \mathcal{P}$ such that $r(P) = r(Q)$ and assume there exists a scalar $\alpha \ge 0$ such that $(\alpha+1)P \le Q+ \alpha I$. Then $P = Q$.
\end{lemma}
\begin{proof}
Let $x \in R(P), \norm{x}=1$ be arbitrary. Then $Px = x$, so 
$$\alpha + 1 = (\alpha+1)\inner{Px}{x} \le \inner{(Q+\alpha I)x}{x} = \inner{Qx}{x} + \alpha \inner{x}{x} = \inner{Qx}{x} + \alpha$$
which implies $\inner{Qx}{x} \ge 1$. By the Cauchy-Schwarz inequality we obtain
$$1 \le \inner{Qx}{x} \le \norm{Qx}\norm{x} \le \norm{Q}\norm{x}^2 \le 1$$
and hence $Qx = \lambda x$ for some $\lambda \in \C$. Since $x \ne 0$, it must be $\lambda \in \{0,1\}$ but it cannot be $\lambda = 0$ because of $\inner{Qx}{x} \ge 1$. Therefore $\lambda = 1$ and hence $Qx = x$ meaning $x \in R(Q)$. It follows that $R(P) \subseteq R(Q)$, so by the equality of ranks we conclude $P = Q$.
\end{proof}

Finally, we present a few auxiliary results from \cite{semrl} which will be used in the latter part of the proof.
\begin{lemma}[\v{S}emrl]\label{tR <=}
Assume that $P,Q,R \in \mathcal{P}_1$ satisfy $P \perp Q$ and $R \le P+Q$. Then
$$\alpha(P,R) := \max\{t \in [0,+\infty\rangle : tR \le Q+2P\} = \frac{2}{2-\Tr(PR)}.$$
\end{lemma}
Note here that if $P = R$, then $$tP \le Q+2P \iff (t-2)P \le Q \stackrel{\iff}{} t-2 \le 0 \iff t \le 2 = \frac{2}{2-\Tr P^2}$$
so $\alpha(P,P) = 2$ does not depend on $Q$. On the other hand, if $P \ne R$, then $Q$ is uniquely determined as an orthogonal projection of rank $1$ with the property $P \perp Q$ and $R \le P+Q$ (i.e. $R-P \le Q$). Hence the notation $\alpha(P,R)$ does not include $Q$.

We also state a stronger version of Uhlhorn's theorem from \cite{EffectAlgebras}:

\begin{proposition}[\v{S}emrl]\label{preserves orthogonality}
Let $H$ be a Hilbert space with $\dim H \ge 3$. Let $\phi : \mathcal{P}_1\to \mathcal{P}_1$ be an injective map such that for every maximal orthogonal subset $S \subseteq \mathcal{P}_1$ the set $\phi(S)$ is also a maximal orthogonal subset of $\mathcal{P}_1$. Then there exists a unitary or antiunitary operator $U : H\to H$ such that $$\phi(P) = UPU^*, \quad \text{for all }P \in \mathcal{P}_1.$$
\end{proposition}

\section{Proof of the main result}

Recall that a continuous map $H_n^{\ge 0} \to H_n^{\ge 0}$ which preserves spectrum in fact preserves the characteristic polynomial; namely, the characteristic polynomial is clearly preserved on the set of matrices with $n$ distinct eigenvalues, which is dense in $H_n^{\ge 0}$. It is interesting to note that continuity can be replaced by the order preserving property:

\begin{proposition}\label{continuity proof}
Let $\phi : H_n^{\ge 0} \to H_n^{\ge 0}$ be a map which preserves order and spectrum. Then $\phi$ is continuous and preserves characteristic polynomial.
\end{proposition}
\begin{proof}
The proof is divided into the next series of claims.
\begin{claim}
$\phi$ preserves characteristic polynomial on $H_n^{> 0}$.
\end{claim}
\begin{proof}
Let $A > 0$ be arbitrary and let $0 < \mu_1 < \cdots < \mu_k$ be its eigenvalues with multiplicities $r_1, \ldots, r_k \ge 1$ respectively. Then $\phi(A)$ has the same eigenvalues $\mu_1, \ldots, \mu_k$. We claim that the respective multiplicities are also the same.

Set $\mu_0 = 0$ and $\mu_{k+1}=+\infty$. Choose any two matrices $B,C > 0$ such that $B \le A \le C$ and for all $1 \le i \le n, 1 \le j \le r_i$ we have that
$$\lambda_{r_1+\cdots + r_{i-1} + 1}(B), \ldots, \lambda_{r_1+\cdots + r_{i-1} + r_i}(B)$$
are distinct elements of $ \langle \mu_{i-1}, \mu_i]$ and that
$$\lambda_{r_1+\cdots + r_{i-1} + 1}(C), \ldots, \lambda_{r_1+\cdots + r_{i-1} + r_i}(C)$$
are distinct elements of $[\mu_i, \mu_{i+1}\rangle$.

For concreteness, by diagonalization for each $1 \le i \le k$ we can pick mutually orthogonal rank-one projections $P_{i1}, \ldots, P_{ir_i} \in \mathcal{P}_1$ such that
$$A = \sum_{i=1}^k \mu_{i}\sum_{j=1}^{r_i} P_{ij}.$$
For $1 \le i \le k$ pick $\varepsilon_i > 0$ such that $\mu_i - (r_i-1)\varepsilon_i > \mu_{i-1}$ and then define
$$B := \sum_{i=1}^k \sum_{j=1}^{r_i} (\mu_i - (r_i-j)\varepsilon_i)P_{ij}.$$
Similarly, for $1 \le i \le k$ pick $\eta_i > 0$ such that $\mu_i + (r_i-1)\eta_i < \mu_{i+1}$ and then define
$$C := \sum_{i=1}^k \sum_{j=1}^{r_i} (\mu_i + (j-1)\eta_i)P_{ij}.$$
Then $B$ and $C$ have $n$ distinct eigenvalues and we have
$$B \le A \le C \implies \phi(B) \le \phi(A) \le \phi(C).$$
We have $k_{\phi(B)} = k_B$ and $k_{\phi(C)} = k_C$ and therefore for each $1 \le i \le k$ and $1 \le j \le r_i$ by the Monotonicity theorem we get
\begin{align*}
    \mu_{i-1} &< \mu_i - (r_i-j)\varepsilon_i \\
    &= \lambda_{r_1 + \cdots + r_{i-1} + j}(\phi(B)) \\
    &\le \lambda_{r_1 + \cdots + r_{i-1} + j}(\phi(A)) \\
    &\le \lambda_{r_1 + \cdots + r_{i-1} + j}(\phi(C)) \\
    &= \mu_i + (j-1)\eta_i \\
    &< \mu_{i+1}
\end{align*}
and hence it has to be $\lambda_{r_1 + \cdots + r_{i-1} + j}(\phi(A)) = \mu_i.$
\end{proof}

\begin{claim}\label{preserves characteristic polynomial}
$\phi$ preserves characteristic polynomial on $H_n^{\ge 0}$.
\end{claim}
\begin{proof}
We know the claim is true for all strictly positive matrices so let $A \ge 0$ be an arbitrary singular matrix. Let $0 =\mu_1 < \mu_2 < \cdots < \mu_k$ be its eigenvalues with multiplicities $r_1, \ldots, r_k \ge 1$. Then $\phi(A)$ has the same eigenvalues $\mu_1, \ldots, \mu_k$; we claim that the respective multiplicities are also the same.

By diagonalization for each $1 \le i \le k$ we can choose mutually orthogonal rank-one projections $P_{i1}, \ldots, P_{ir_i}$ such that
$$A = \sum_{i=1}^k \mu_{i}\sum_{j=1}^{r_i} P_{ij} = \sum_{i=2}^k \mu_{i}\sum_{j=1}^{r_i} P_{ij}.$$
For any $\varepsilon \in \langle 0, \mu_2\rangle$ set
$$C:= \varepsilon \sum_{j=1}^{r_1} P_{1j} + \sum_{i=2}^k \mu_i\sum_{j=1}^{r_i} P_{ij} = \varepsilon \sum_{j=1}^{r_1} P_{1j} + A.$$
Then $C > 0$ and hence by the previous claim it follows that $k_{\phi(C)} = k_C$. Since $A \le C$, by the Monotonicity theorem for all $1 \le j \le r_1$ we have
$$\lambda_{j}(\phi(A)) \le \lambda_j(\phi(C)) = \varepsilon.$$
It follows that $\lambda_{1}(\phi(A)) = \cdots = \lambda_{r_1}(\phi(A)) = 0$. Therefore, the multiplicity of $0$ as an eigenvalue of $\phi(A)$ is $\ge r_1$.
Moreover, for all $2 \le i \le k$ and $1 \le j \le r_i$ we have
$$\lambda_{r_1 + \cdots + r_{i-1} + j}(\phi(A)) \le \lambda_{r_1 + \cdots + r_{i-1} + j}(\phi(C)) = \mu_i.$$

Now for $2 \le i \le k$ choose $\varepsilon_i > 0$ such that $\mu_i - (r_i-1)\varepsilon_i > \mu_{i-1}$ and then define
$$B := \sum_{i=2}^k \sum_{j=1}^{r_i} (\mu_i - (r_i-j)\varepsilon_i)P_{ij}.$$
Then the multiplicity of $0$ as an eigenvalue of $\phi(B)$ is at least the multiplicity of $0$ as an eigenvalue of $B$, while all the other eigenvalues of $\phi(B)$ are necessarily equal to the ones of $B$ by virtue of being distinct. We conclude $k_{\phi(B)} = k_B$. Since $B \le A$, the Monotonicity theorem implies that for all for all $2 \le i \le k$ and $1 \le j \le r_i$ we have
$$\mu_{i-1}<\mu_i - (r_i-1)\varepsilon_i = \lambda_{r_1 + \cdots + r_{i-1} + j}(\phi(B)) \le \lambda_{r_1 + \cdots + r_{i-1} + j}(\phi(A)) \le \mu_i$$
and therefore it follows $\lambda_{r_1 + \cdots + r_{i-1} + j}(\phi(A)) = \mu_i$. This completes the proof of the claim.
\end{proof}

\begin{claim}\label{trace, rank and orthogonal projections}
$\phi$ preserves trace, rank and orthogonal projections.
\end{claim}
\begin{proof}
This follows directly from Claim \ref{preserves characteristic polynomial}.
\end{proof}

\begin{claim}\label{right limit}
For every $A \in H_n^{\ge 0}$ and $\varepsilon > 0$ holds $\norm{\phi\left(A + \varepsilon I\right) - \phi(A)} \le n\varepsilon$.
\end{claim}
\begin{proof}
Since $\phi\left(A + \varepsilon I\right) - \phi(A)\ge 0$ we have
$$\norm{\phi\left(A + \varepsilon I\right) - \phi(A)} = \nu\left(\phi\left(A + \varepsilon I\right) - \phi(A)\right) \le \Tr\left(\phi\left(A + \varepsilon I\right) - \phi(A)\right) = \Tr \left(\varepsilon I\right) = n\varepsilon.$$
\end{proof}
\begin{claim}\label{left limit}
For every $A \in H_n^{\ge 0}$ and $\varepsilon \in \langle 0,\lambda_{\mathrm{min}}(A)\rangle$ holds $\norm{\phi(A) - \phi\left(A - \varepsilon I\right)} \le n\varepsilon$.
\end{claim}
\begin{proof}
Since $\phi(A) - \phi\left(A - \varepsilon I\right)\ge 0$, we can apply the same argument as in the proof of Claim \ref{right limit}.
\end{proof}

\begin{claim}\label{continuous on >0}
$\phi$ is continuous at every $A > 0$.
\end{claim}
\begin{proof}
Let $A > 0$ be arbitrary and choose any $\varepsilon \in \langle 0,\lambda_{\mathrm{min}}(A)\rangle$. Assume $C \ge 0$ satisfies $\norm{C-A} \le \frac\varepsilon{n}$. The latter is equivalent to
$$A -\frac{\varepsilon}n I \le C\le A + \frac{\varepsilon}n I$$
and then (by the order preserving property) we also have
$$\phi\left(A - \frac{\varepsilon}n I\right) \le \phi(C) \le \phi\left(A + \frac{\varepsilon}n I\right).$$
By Claims \ref{right limit} and \ref{left limit}, this implies
$$-\varepsilon I\le \phi\left(A - \frac{\varepsilon}n I\right)-\phi(A) \le \phi(C)-\phi(A) \le \phi\left(A + \frac{\varepsilon}n I\right)-\phi(A) \le \varepsilon I$$
and hence
$$\norm{\phi(C) - \phi(A)} \le \varepsilon.$$
Since the choice of such $C$ was arbitrary, we conclude that $\phi$ is continuous at $A$.
\end{proof}

\begin{claim}
$\phi$ is globally continuous, i.e. $\phi$ is continuous at every $A \ge 0$.
\end{claim}
\begin{proof}
Let $A \ge 0$ be arbitrary and let $\varepsilon > 0$. For $X \in H_n^{\ge 0}$ denote $X_\varepsilon := X + \frac\varepsilon{3n}I$. Since $A_\varepsilon > 0$, we know that $\phi$ is continuous at $A_\varepsilon$. Hence
$$(\exists \delta > 0)(\forall B \ge 0)\qquad \norm{B-A_\varepsilon} \le \delta \implies \norm{\phi(B) - \phi(A_\varepsilon)} \le \frac\varepsilon3.$$
Then for all $C \ge 0$ such that $\norm{C-A} \le \delta$ we have
$$\norm{C_\varepsilon - A_\varepsilon} = \norm{C-A} \le \delta$$
and therefore by Claim \ref{right limit}
$$\norm{\phi(C) - \phi(A)} \le \norm{\phi(C) - \phi(C_\varepsilon)} + \norm{\phi(C_\varepsilon) - \phi(A_\varepsilon)} + \norm{\phi(A_\varepsilon) - \phi(A)} \le \frac\varepsilon3 + \frac\varepsilon3 + \frac\varepsilon3 = \varepsilon.$$
We conclude that $\phi$ is continuous at $A$.
\end{proof}
\end{proof}

We now proceed with the proof of Theorem \ref{main result}.

\setcounter{section}{1}
\setcounter{proposition}{3}
\setcounter{claim}{0}

\begin{proof}[Proof of Theorem 1.3.]

The proof will follows from the next series of claims.

\begin{claim}\label{preserves projections}
$\phi$ preserves rank (and hence the characteristic polynomial) on orthogonal projections.
\end{claim}
\begin{proof}
Every orthogonal projection is a part of a chain $P_1 < P_2 < \cdots < P_n$ where $r(P_j) = j$ for all $1 \le j \le n$. For such a chain, by the order preserving property we have
$$\phi(P_1) \le \phi(P_2) \le \cdots \le \phi(P_n).$$
By injectivity of $\phi$ all the inequalities have to be strict. In particular, all ranks have to be distinct and therefore $r(\phi(P_j)) = j$ for all $1 \le j \le n$.
\end{proof}

\begin{claim}\label{homogeno na projektorima}
For every orthogonal projection $P \in \mathcal{P}$ and scalar $\alpha \ge 0$ we have $\phi(\alpha P) = \alpha\phi(P)$.
\end{claim}
\begin{proof}
We can assume $\alpha > 0$. The map $\frac1\alpha \phi(\alpha \cdot) : H_n^{\ge 0} \to H_n^{\ge 0}$ satisfies all assumptions of Theorem \ref{main result} so by Claim \ref{preserves projections} there exists an orthogonal projection $Q \in \mathcal{P}$ of rank $r(P)$ such that $\phi(\alpha P) =\alpha Q$.

Assume $\alpha > 1$ (the case $\alpha <1$ is analogous). From $P \le \alpha P$ it follows $\phi(P) \le \phi(\alpha P) = \alpha Q$ so in particular $R(\phi(P)) \le R(\alpha Q) = R(Q)$ which implies $Q = \phi(P)$.
\end{proof}


Let $\alpha \ge 0$. Consider the map
$$\phi_\alpha : H_n^{\ge 0} \to H_n^{\ge 0}, \quad \phi_\alpha(A) := \phi(A+\alpha I)-\alpha I.$$
Then $\phi_\alpha$ is also an injective map which preserves order and shrinks spectrum.

\begin{claim}\label{plus scalar times identity}
Let $P \in \mathcal{P}$ be an orthogonal projection and $\alpha \ge 0$. Then $\phi(P+\alpha I) = \phi(P)+\alpha I$.
\end{claim}
\begin{proof}
The matrix 
$$\phi_\alpha (P) = \phi(P+\alpha I)-\alpha I$$
is an orthogonal projection of rank $r(P)$. Therefore, there is an orthogonal projection $Q(\alpha) \in \mathcal{P}$ of rank $r(P)$ such that
$$\phi(P+\alpha I) = Q(\alpha)+\alpha I.$$
In particular, we have
$$(\alpha+1)P = P+\alpha P \le P+\alpha I \implies (\alpha+1)\phi(P) = \phi((\alpha+1)P) \le \phi(P+\alpha I) = Q(\alpha)+\alpha I.$$
Lemma \ref{projection equality} now implies $Q(\alpha) = \phi(P)$. In particular, we have $\phi_\alpha(P) = \phi(P)$ for all $\alpha \ge 0$.
\end{proof}

\begin{claim}\label{linearity on projections}
Let $P \in \mathcal{P}$ be an orthogonal projection and $\alpha,\beta \ge 0$. Then $\phi(\alpha P+\beta I) = \alpha \phi(P)+\beta I$.
\end{claim}
\begin{proof}
According to Claims \ref{plus scalar times identity} and \ref{homogeno na projektorima} we have
$$\alpha \phi(P) = \alpha \phi_\beta (P) = \phi_\beta (\alpha P) = \phi(\alpha P+\beta I) - \beta I$$
which implies
$$\phi(\alpha P+\beta I) = \alpha \phi(P)+\beta I.$$
\end{proof}

\begin{claim}
$\phi$ preserves characteristic polynomial on matrices with exactly two distinct eigenvalues.
\end{claim}
\begin{proof}
Let $A \ge 0$ be a matrix with exactly two distinct eigenvalues. Let $P$ be the $\lambda_{\mathrm{max}}(A)$-eigenprojection. Then we have
$$A = \lambda_{\mathrm{min}}(A)(I-P) + \lambda_{\mathrm{max}}(A)P = \lambda_{\mathrm{min}}(A) I + (\lambda_{\mathrm{max}}(A)-\lambda_{\mathrm{min}}(A))P$$
and hence by Claim \ref{linearity on projections} we obtain
\begin{align*}
    \phi(A) &= \phi(\lambda_{\mathrm{min}}(A) I + (\lambda_{\mathrm{max}}(A)-\lambda_{\mathrm{min}}(A)) P) \\
    &= \lambda_{\min}(A) I + (\lambda_{\mathrm{max}}(A)-\lambda_{\mathrm{min}}(A)) \phi(P) \\
    &= \lambda_{\mathrm{min}}(A)(I-\phi(P)) + \lambda_{\mathrm{max}}(A)\phi(P).
\end{align*}
As $\phi(P)$ is an orthogonal projection of rank $r(P)$, the claim follows.
\end{proof}

\begin{claim}
For every $\alpha \ge 0$ it holds $\phi(\alpha I) = \alpha I$.
\end{claim}
\begin{proof}
We have $\sigma(\phi(\alpha I)) \subseteq \{\alpha\}$ so that $\sigma(\phi(\alpha I)) = \{\alpha\}$ and hence the claim follows.
\end{proof}

\begin{claim}
$\phi$ preserves characteristic polynomial of every $A \ge 0$.
\end{claim}
\begin{proof}
Note that the claim is already proven if $k=1,2$ so we can assume $k \ge 3$.

Let $A \ge 0$ be arbitrary with eigendecomposition $$A = \sum_{i=1}^k \mu_iP_i, \qquad \mu_1 < \cdots < \mu_k.$$ For $2 \le i \le k$ set
$$B_i := \mu_1(P_1+\cdots + P_{i-1}) + \mu_i(P_i+\cdots + P_k)$$
so that $B_i \le A$. Similarly, for $1 \le i \le k-1$ set
$$C_i := \mu_i(P_1+\cdots + P_i) + \mu_k(P_{i+1}+\cdots + P_k)$$
so that $A \le C_i$.

For all $2 \le i \le k-1$ we have $\phi(B_i) \le \phi(A) \le \phi(C_i)$ and therefore for all $1 \le j \le r_i$ holds
$$\mu_i = \lambda_{r_1+\cdots+r_{i-1}+j}(\phi(B_i)) \le \lambda_{r_1+\cdots+r_{i-1}+j}(\phi(A)) \le \lambda_{r_1+\cdots+r_{i-1}+j}(\phi(C_i)) = \mu_i.$$
We have $\phi(A) \le \phi(C_1)$ and hence for all $1 \le j \le r_1$ we have
$$\lambda_{j}(\phi(A)) \le \lambda_{j}(\phi(C_1)) \le \mu_1$$
We have $\phi(B_k) \le \phi(A)$ and hence for all $1 \le j \le r_k$ we have
$$\mu_n = \lambda_{r_1+\cdots+r_{k-1}+j}(\phi(B_n)) \le \lambda_{r_1+\cdots+r_{k-1}+j}(\phi(A)).$$
The claim follows.
\end{proof}

\begin{claim}
$\phi$ is globally continuous.
\end{claim}
\begin{proof}
This follows from Proposition \ref{continuity proof}.
\end{proof}



As the set $H_n^{> 0}$ is open in $H_n^{\ge 0}$, by the invariance of domain theorem applied to the map $\phi|_{H_n^{> 0}} : H_n^{> 0} \to \phi(H_n^{> 0})$ , we get that its range $\phi(H_n^{> 0})$ is an open set containing the identity $I$. In particular, there exists $\varepsilon > 0$ such that $\overline{B}(I,\varepsilon) \subseteq \phi(H_n^{> 0})$ where $\overline{B}(I,\varepsilon)$ denotes the closed ball around $I$ of radius $\varepsilon$.

\begin{claim}
The restriction $\phi|_{\mathcal{P}_k} : \mathcal{P}_k \to \mathcal{P}_k$ is bijective for all $1 \le k \le n$.
\end{claim}
\begin{proof}
By Claim \ref{trace, rank and orthogonal projections} we already know $\phi(\mathcal{P}_k) \subseteq \mathcal{P}_k$. Since $\phi$ is injective by assumption, it remains to prove surjectivity. Let $Q \in \mathcal{P}_k$ be arbitrary. Then
$$\norm{(\varepsilon Q + I) - I} = \varepsilon\norm{Q} = \varepsilon \implies \varepsilon Q + I \in \phi(H_n^{\ge 0})$$
so there exists $A \in H_n^{\ge 0}$ such that $\varepsilon Q + I = \phi(A)$. The spectrum of $A$ satisfies $\sigma(A) \subseteq  \{1+\varepsilon,1\}$ so there exists $P \in \mathcal{P}$ such that $A = \varepsilon P + I$. It follows
$$\varepsilon Q + I = \phi(A) = \phi(\varepsilon P + I) = \varepsilon \phi(P) + I$$
so $\phi(P) = Q$. Since $\phi$ preserves rank, it has to be $r(P) = r(Q) = k$ so $P \in \mathcal{P}_k$.
\end{proof}

The proof can be now completed using the arguments  given in \cite{semrl}. For completeness, we include details:

\begin{claim}\label{expands trace of product}
For all $P,R \in \mathcal{P}_1$ we have
$$\Tr(\phi(P)\phi(R)) \ge \Tr(PR).$$
\end{claim}
\begin{proof}
If $P=R$, both sides are equal to $1$ so assume $P \ne R$. Then there exists a unique $S \in \mathcal{P}_2$ such that $P,R \le S$. Set $Q := S-P$ so that $Q \in \mathcal{P}_1$ and $P \perp Q$ i $R \le P+Q$.

We define $A:= \phi(Q+2P)$. Then $r(A) = 2$. The spectrum of $A$ satisfies $$\sigma(A) = \sigma(Q+2P) \implies \{1,2\} \subseteq \sigma(A) \subseteq \{0,1,2\}$$
and both eigenvalues $1$ and $2$ have multiplicities equal to one.

Hence, we have the spectral decomposition $A = T_1 + 2T_2$ where $T_1,T_2 \in \mathcal{P}_1$ is a perpendicular pair of projections. Furthermore, we have
$$2\phi(P) = \phi(2P) \le \phi(Q+2P) = A = T_1+2T_2 \le 2I$$
which implies $A|_{R(\phi(P))} = 2I$. Therefore, for all nonzero $x \in R(\phi(P))$ we have
$$(T_1 + 2T_2)x=2x$$
which implies that $x$ is an element of the one-dimensional eigenspace associated to the eigenvalue $2$, which is equal to $R(T_2)$.
It follows $R(\phi(P)) \subseteq R(T_2)$. From $r(T_2) = r(\phi(P)) = 1$ we have the equality of ranges so $T_2 = \phi(P)$.

Therefore, $$A = T_1 + 2\phi(P)$$
where $T_1 \in \mathcal{P}_1$ is perpendicular to $\phi(P) \in \mathcal{P}_1$. Set $$b := \frac{2}{2-\Tr(PR)}.$$
Lemma \ref{tR <=} implies $bR \le Q+2P$ and then
$$b\phi(R) \le T_1 + 2\phi(P).$$
In particular, 
$$R(\phi(R)) = R( b\phi(R)) \subseteq R(T_1 + 2\phi(P)) = R(T_1) \oplus R(\phi(P)) = R(T_1 + \phi(P)),$$
which is equivalent to $\phi(R) \le T_1+\phi(P)$. Another application of Lemma \ref{tR <=} gives
$$\frac2{2-\Tr(PR)} = b \le \alpha(\phi(P),\phi(R)) = \frac2{2-\Tr(\phi(P)\phi(R))}$$
which immediately implies $\Tr(PR) \le \Tr(\phi(P)\phi(R)).$
\end{proof}

\begin{claim}
For all $P_1,P_2 \in \mathcal{P}_1$ we have the implication
$$\phi(P_1) \perp \phi(P_2) \implies P_1 \perp P_2.$$
\end{claim}
\begin{proof}
By Claim \ref{expands trace of product} we have
$$0 \le \Tr(P_1P_2) \le \Tr(\phi(P_1)\phi(P_2)) = 0$$
so $\Tr(P_1P_2) = 0$ which directly implies $P_1 \perp P_2$.
\end{proof}

Now assume $n \ge 3$. The map $(\phi|_{\mathcal{P}_1})^{-1} : \mathcal{P}_1 \to \mathcal{P}_1$ satisfies the assumptions of Proposition \ref{preserves orthogonality} so there exists a unitary or antiunitary map $U : \C^n \to \C^n$ such that
$$(\phi|_{\mathcal{P}_1})^{-1}(P) = UPU^*, \quad \text{ for all } P \in \mathcal{P}_1.$$
In other words, we have $\phi(UPU^*) = P$ for all $P \in \mathcal{P}_1$.

Without loss of generality, by considering the map $\phi(U\cdot U^*)$ (which is again injective, preserves order and shrinks spectrum) we can assume that $\phi(P) = P$ for all $P \in \mathcal{P}_1$. We claim that for all $A \in H_n^{\ge 0}$ holds $\phi(A) = A$.

First we consider matrices of the form $A = tI -cP$ where $P \in \mathcal{P}_1$, $t > 0$ and $c \in [0,t]$.

Let $Q \in \mathcal{P}_1$ be such that $Q \perp P$. Then
$$tQ \le tI - cP \le tI$$
so by applying $\phi$ we also get
$$tQ \le \phi(tI -cP) \le tI.$$

This implies $\phi(tI-cP)|_{R(Q)} = tI$. Since $Q$ was arbitrary, it follows $\phi(tI-cP)|_{R(P)^\perp} = tI$.
We have
$$\sigma(\phi(tI-cP)) = \sigma(tI-cP) = \{t,t-c\}$$
and $t-c$ has multiplicity equal to one. We therefore get the spectral decomposition
$$\phi(tI-cP) = t(I-P)+(t-c)P = tI-cP$$
which is precisely what we wanted to show.

Now take an arbitrary $A \in H_n^{\ge 0}$. Then we have the spectral decomposition
$$A = \sum_{j=1}^k t_jP_j$$
for some distinct $t_1, \ldots, t_k > 0$ and nonzero pairwise perpendicular projections $P_1, \ldots, P_k \in \mathcal{P}$. We claim that for all $1 \le j \le k$ and $x \in R(P_j)$ we have $\phi(A)x = t_jx$, while for all $x \in \C^n, x \perp R(P)$ we have $\phi(A)x=0$ where $P:= P_1+\cdots+P_k$.

To prove the first part, we fix $x \in R(P_j)$ and let $Q \in \mathcal{P}_1$ whose image is spanned by $x$. If we set $t = \max_{1 \le j \le k} t_j$, we get
\begin{align*}
    t_j Q &\le t_jP \le \sum_{i=1}^k t_iP_i \le \sum_{i=1}^k t_iP_i+(t-t_j)\underbrace{(P_j-Q)}_{\ge 0} \\
    &= \sum_{\substack{i=1 \\ i\ne j}}^k t_iP_i + tP_j - (t-t_j)Q \le  t\sum_{i=1}^k P_i - (t-t_j)Q \le tI-(t-t_j)Q
\end{align*}
or
$$t_jQ \le A \le tI-(t-t_j)Q.$$
This implies
$$t_jQ = \phi(t_jQ) \le \phi(A) \le \phi(tI - (t-t_j)Q) = tI - (t-t_j)Q,$$
so we conclude $\phi(A)x = t_jx$.

Now take $x \perp R(P)$, and let $R \in \mathcal{P}_1$ whose image is spanned by $x$. Then we have
$$0 \le A = \sum_{i=1}^k t_iP_i \le t\sum_{i=1}^k P_i \le t(I-R)$$
which implies
$$0 \le \phi(A) \le \phi(t(I-R)) = t(I-R)$$
and therefore $\phi(A)x = 0$.

This proves the theorem in case of $n \ge 3$. Since the case $n=1$ is trivial, assume that $n = 2$. Recall that $\phi|_{\mathcal{P}_1} :\mathcal{P}_1 \to \mathcal{P}_1$ is bijective and that for all $P,R \in \mathcal{P}_1$ we have $$\Tr(\phi(P)\phi(R)) \ge \Tr\phi(PR).$$
We claim that the equality holds. Suppose to the contrary that there exist $P,R \in \mathcal{P}_1$ such that $$\Tr(\phi(P)\phi(R)) > \Tr\phi(PR).$$
Choose unitary matrices $U,V \in U_n$ such that
$$P = UE_{11}U^*, \quad \phi(P) = VE_{11}V^*.$$
Replacing $\phi, P$ and $R$ by $V^*\phi(U\cdot U^*)V, E_{11}$ and $U^*RU$ respectively, we can assume that $$\phi(P) = P = E_{11}.$$
There exist $a,b \in [0,1]$ such that 
$$R = \begin{bmatrix} a & * \\ * & 1-a\end{bmatrix}, \quad \phi(R) = \begin{bmatrix} b & * \\ * & 1-b\end{bmatrix}.$$
and because of $\Tr(\phi(P)\phi(R)) > \phi(PR)$ it follows $b > a$. By bijectivity of $\phi|_{\mathcal{P}_1}$ it follows that there exists $Q \in \mathcal{P}_1$ such that $\phi(Q) = E_{22}$. We have $$\phi(P) = E_{11} \perp E_{22} = \phi(Q) \implies P \perp Q$$
which implies $Q = E_{22}$. But then
$$\Tr(\phi(Q)\phi(R)) = 1-b < 1-a = \Tr(QR)$$
which is a contradiction.
\end{proof}

\setcounter{section}{3}
\setcounter{proposition}{1}

As a consequence of Theorem \ref{main result} we obtain the result for maps on hermitian matrices as well:

\begin{corollary}\label{hermitian result}
Let $\phi : H_n \to H_n$ be an injective map which preserves order and shrinks spectrum. Then there exists a unitary or antiunitary map $U : \C^n \to \C^n$ such that $\phi(A) = UAU^*$ for all $A \in H_n$.
\end{corollary}
\begin{proof}
For arbitrary $\alpha \ge 0$ we consider the map 
$$\psi_\alpha : H_n^{\ge 0} \to H_n^{\ge 0}, \quad \psi_\alpha(A) := \phi(A-\alpha I) + \alpha I.$$
The map $\psi_\alpha$ clearly satisfies all assumptions of Theorem \ref{main result}. It follows that there is a unitary or antiunitary map $U_\alpha : \C^n \to \C^n$ such that
$$\psi_\alpha(A) = U_\alpha AU_\alpha^*, \quad \text{ for all } A \ge 0.$$
For all $A \ge 0$ we therefore have
$$\phi(A-\alpha I) = \psi_\alpha (A)-\alpha I = U_\alpha AU_\alpha^* - \alpha I = U_\alpha(A-\alpha I)U_\alpha^*.$$
It follows that
$$\phi(A) = U_\alpha AU_\alpha^*, \quad \text{ for all } A \in H_n, A \ge -\alpha I.$$
Now fix $\alpha \ge 0$. For every $A \ge 0$ trivially we have $A \ge -\alpha I$ and hence
$$U_\alpha AU_\alpha^* = \phi(A) = U_0AU_0^* \implies U_0^*U_\alpha A = A U_0^*U_\alpha.$$
Therefore, $U_0^*U_\alpha$ commutes with all positive matrices and hence is of the form $\lambda I$ for some $\lambda \in \Sph^1$. It follows that $U_\alpha = \lambda U_0$, so without loss of generality we may assume $U_\alpha = U_0$.

For every $A \in H_n$ holds $A \ge -\norm{A}I$ and hence
$$\phi(A) = U_{\norm{A}}AU_{\norm{A}}^* = U_{0}AU_{0}^*$$
which proves the claim.
\end{proof}

\section{Counterexamples}

In this section, when $n \ge 2$, we show that none of the assumptions in the Theorem \ref{main result} can be omitted, even if we replace spectrum shrinking by spectrum preserving and add continuity as an additional assumption.

Evidently, when $n=1$, all assumptions are superfluous except spectrum shrinking.

\begin{example}\label{sorted spectrum}
Let $n \ge 2$ and consider the map $\phi : H_n \to H_n$ which maps a matrix $A \in H_n$ to
$$\phi(A) := \diag(\lambda_1(A), \ldots, \lambda_n(A))$$
where $\lambda_1(A) \le \cdots \le \lambda_n(A)$ are all eigenvalues of $A$. Then $\phi$ is continuous, order and spectrum preserving, but not injective. Its restriction $\phi|_{H_n^{\ge 0}} : H_n^{\ge 0} \to H_n^{\ge 0}$ obviously has the same properties.

The order preserving property is precisely Theorem \ref{monotonicity theorem}.
\end{example}

\begin{example}
The map $\phi : H_n^{\ge 0} \to H_n^{\ge 0}$ given by $\phi(A) = A^{1/2}$ is injective, continuous and order preserving but is not spectrum shrinking.
\end{example}

\begin{example}
Let $n \ge 2$. For a matrix $A \in H_n$ denote its set of eigenvalues by $\lambda_1(A) \le \cdots \le \lambda_n(A)$, as usual. Define
$$\mathscr{U} := \{A \in H_n^{\ge 0} : \lambda_j(A) \in \langle j-1,j+1\rangle\}.$$
Notice that $\mathscr{U}$ is an open set in $H_n^{\ge 0}$. Namely, recall that the map $F : H_n^{\ge 0} \to \R^n, A \mapsto \left(\lambda_1(A), \ldots, \lambda_n(A)\right)$ is continuous and $\mathcal{U}$ is precisely the preimage of the open set $\prod_{j=1}^n \langle j-1,j+1\rangle$ by the map $F$.

For any $A \in H_n^{\ge 0}$ define the unitary matrix $W_A \in U_n$ as
$$W_A := \begin{cases}
I, &\quad\text{ if }A \notin \mathscr{U},\\
\diag\left(\begin{bmatrix} \cos(2\pi r_A) & -\sin(2\pi r_A) \\ \sin(2\pi r_A) & \cos(2\pi r_A)\end{bmatrix}, 1, \ldots, 1\right), &\quad\text{ if }A \in \overline{\mathscr{U}}.\\
\end{cases}$$
where $r_A = \max_{1 \le j \le n} \abs{\lambda_j(A)-j} \in [0,1]$. This definition makes sense because
$$\overline{\mathscr{U}} := \{A \in H_n^{\ge 0} : \lambda_j(A) \in [j-1,j+1]\}$$
and for $A \in \overline{\mathscr{U}}\setminus \mathscr{U}$ we have $r_A = 1$ and hence $W_A= I$. Moreover, the map $H_n^{\ge 0} \to U_n, A \mapsto W_A$ is continuous by the pasting lemma. Notice also that $W_A$ depends only on the characteristic polynomial of $A$.    

Now define the map $$\phi : H_n^{\ge 0} \to H_n^{\ge 0}, \quad \phi(A) := W_A A W_A^*, \quad A \ge 0.$$
Then $\phi$ is clearly continuous and spectrum preserving. We claim that $\phi$ is also injective. Indeed, let $A,B \in H_n^{\ge 0}$ such that $\phi(A) = \phi(B)$. In particular we have the equality of characteristic polynomials $k_A = k_B$, so $W_A = W_B$ as well and hence
$$\phi(A) = \phi(B) \implies W_A A W_A^* = W_A B W_A^* \implies A= B.$$
However, $\phi$ does not preserve order. Indeed, if we define
$$A= \diag(1,2,\ldots,  n-1, n), \quad B = \diag\left(1,2, \ldots, n-1, n+\frac34\right)$$
Then $A,B \in \mathscr{U}$ and $r_A = 0, r_B = \frac34,$ so
$$\phi(A) = A, \quad \phi(B) = \diag\left(2,1,3, \ldots, n-1, n+\frac34\right).$$
We have $A \le B$ but $\phi(A) \not\le \phi(B)$.
\end{example}

\end{document}